\documentclass[reqno, oneside, 11pt]{amsart}
\usepackage[letterpaper]{geometry}
\usepackage{appendix}
\usepackage{amsmath}
\usepackage{amssymb}
\usepackage{enumerate,enumitem}
\usepackage{verbatim}
\usepackage{mathrsfs}
\usepackage{mathtools}
\usepackage{hyperref}
\usepackage{amsthm, mathdots, multicol,placeins}

\DeclareMathOperator{\Jac}{Jac}
\DeclareMathOperator{\USp}{USp}
\DeclareMathOperator{\SU}{SU}
\DeclareMathOperator{\Sp}{Sp}
\DeclareMathOperator{\U}{U}

\DeclareMathOperator{\ST}{ST}
\DeclareMathOperator{\Gal}{Gal}
\DeclareMathOperator{\End}{End}

\DeclareMathOperator{\diag}{diag}
\DeclareMathOperator{\TL}{TL}
\DeclareMathOperator{\AST}{AST}
\DeclareMathOperator{\LL}{L}
\DeclareMathOperator{\Tr}{Tr}

\newcommand{\Z}{\mathbb{Z}}
\newcommand{\Q}{\mathbb{Q}}
\newcommand{\C}{\mathbb{C}}

\newtheorem{theorem}{Theorem}[section]
\newtheorem{example}[theorem]{Example}

\newtheorem{proposition}[theorem]{Proposition}

\newtheorem{lemma}[theorem]{Lemma}
\newtheorem{definition}[theorem]{Definition}

\newtheorem{remark}[theorem]{Remark}

\author{Heidi Goodson}
\address{Department of Mathematics, Brooklyn College, City University of New York; 2900 Bedford Avenue, Brooklyn, NY 11210 USA}
\email{heidi.goodson@brooklyn.cuny.edu}

\author{Rezwan Hoque}
\address{Department of Mathematics, Brooklyn College, City University of New York; 2900 Bedford Avenue, Brooklyn, NY 11210 USA}
\email{rezwan.hoque98@bcmail.cuny.edu}

\title{Sato-Tate Groups and Distributions of $y^\ell=x(x^\ell-1)$}

\begin{document}

\begin{abstract}
 Let $C_\ell/\mathbb Q$ denote the nonsingular curve with affine model $y^\ell=x(x^\ell-1)$, where $\ell\geq 3$ is prime. In this paper we study the limiting distributions of the normalized $L$-polynomials of the curves by computing their Sato-Tate groups and distributions. We also provide results for the number of points on the curves over finite fields, including a formula in terms of Jacobi sums when the field $\mathbb F_q$ satisfies $q\equiv 1 \pmod{\ell^2}$.
\end{abstract}

\maketitle 


\section{Introduction}

Let $C_\ell/\mathbb Q$ denote the curve with affine model $y^\ell=x(x^\ell-1)$, where $\ell\geq 3$ is prime. The main goal of this paper is to study the limiting distributions of the normalized $L$-polynomials of the Jacobians of curves in this family. These curves have complex multiplication by $\Q(\zeta_{\ell^2})$, and so work of Johansson \cite{Johansson2017} implies that we can study these limiting distributions by computing the Sato-Tate groups of the Jacobians. Furthermore, work of Aoki in \cite{Aoki2002} shows that the Jacobians of these curves are nondegenerate, which makes this family of curves a natural example to study in the context of Sato-Tate groups.

Recall that the Sato-Tate conjecture, posed in the 1960s by Mikio Sato and John Tate (independently), is a statistical conjecture regarding the distributions of the normalized traces of Frobenius on elliptic curves without complex multiplication (CM). The conjecture has been proven for non-CM elliptic curves defined over totally real fields (see, for example, \cite{Clozel2008,Taylor2008}), and similar results are known for elliptic curves with CM. This conjecture was generalized to higher dimensional abelian varieties by Serre \cite{SerreNXP}. The generalized Sato–Tate conjecture for an abelian variety predicts the existence of a compact Lie group, referred to in the literature as the \emph{Sato-Tate group}, that determines the limiting distribution of coefficients of the normalized $L$-polynomials. While the conjecture is largely still open, it is known to be true in many cases, including for nongeneric abelian surfaces (see \cite{FiteSuthlerland2014,Johansson2017, Taylor2020}) and for abelian varieties with complex multiplication (see \cite[Proposition 16]{Johansson2017}). Even in cases where the conjecture has not yet been proven, there has been a great amount of interest in computing Sato-Tate groups. For example, there are classification results in dimension 2 \cite{Fite2012} and dimension 3 \cite{fitekedlayasuth2023} that determine all possible Sato-Tate groups that correspond to abelian varieties in these dimensions. We refer the reader to \cite{Banaszak2015,fitekedlayasuth2023,GoodsonCatalan,SutherlandNotes} for more background information on Sato-Tate groups and the generalized Sato-Tate conjecture.

Many of the results in the literature in this area are for specific families of abelian varieties with CM since the generalized Sato-Tate conjecture is known to be true in these cases (see, for example, \cite{Arora2016,EmoryGoodson2022, EmoryGoodson2024, FGL2016,FiteLorenzo2018,GalleseGoodsonLombardo2,GoodsonCatalan, GoodsonDegeneracy, LarioSomoza2018}). Most of these results are for nondegenerate abelian varieties since, due to work of Banaszak and Kedlaya \cite{Banaszak2015}, there are well-established techniques for computing Sato-Tate groups in these cases (see Section \ref{sec:nondegeneracy} for more details). However, progress was recently made in \cite{GalleseGoodsonLombardo2,GoodsonDegeneracy} on computing Sato-Tate groups for degenerate abelian varieties.

In this paper, we compute the Sato-Tate groups of Jacobians of the curves $C_\ell: y^\ell=x(x^\ell-1)$ in order to gain information about the limiting distributions of the coefficients of their normalized $L$-polynomials
\begin{equation}\label{eqn:normalizedLpolynomial}
    \overline L_{\mathfrak p}(C_\ell,T):=L_{\mathfrak p}(C_\ell,T/\sqrt{N(\mathfrak{p})})=T^{2g}+a_1T^{2g-1}+a_2T^{2g-2}+\cdots +a_2T^2+a_1T+1,
\end{equation}
where $g$ is the genus of the curve and $L_{\mathfrak p}(C_\ell,T)$ is the $L$-polynomial appearing in the numerator of the zeta function of the curve for primes $\mathfrak p$ of good reduction. Our work generalizes the results of \cite{LarioSomoza2018} which studies the case of the Picard curve with $\ell=3.$ The Jacobians of the curves $C_\ell$ are known to be simple and nondegenerate due to work of Aoki in \cite{Aoki2002}. Furthermore, the curves and their Jacobians have complex multiplication, and so the generalized Sato-Tate conjecture is known to be true for this family. Moreover, the field of definition of the endomorphisms of $\Jac(C_\ell)$ is $K=\Q(\zeta_{\ell^2})$, and so both $\Gal(K/\Q)$ and the component group of the Sato-Tate group are cyclic. Thus, we obtain a relatively easy-to-state result for the Sato-Tate group in Theorem \ref{theorem:STfullgroup}. On the other hand, we encounter some interesting computational challenges, described in Section \ref{sec:momentstats}, due to the high genera of the curves: $g=\ell(\ell-1)/2$, which is quite large even for small primes such as $\ell=5$ and 7.

\subsection*{Organization of the paper}
In Section \ref{sec:curveprelim} we present several important facts about the curves and their Jacobians. In Section \ref{sec:pointcount} we prove  new results for the number of points on the curves over finite fields, including a formula in terms of Jacobi sums when the field $\mathbb F_q$ satisfies $q\equiv 1 \pmod{\ell^2}$. We provide some background information on nondegeneracy in Section \ref{sec:nondegeneracy}.  Section \ref{sec:satotategroupresults} contains the main result of the paper (Theorem \ref{theorem:STfullgroup}), where we give explicit generators of the Sato-Tate group. In Section \ref{sec:momentstats} we describe our methods for computing moment statistics for the curves $C_\ell$ and their Sato-Tate groups. Throughout the paper, we work through the example $\ell=5$ to demonstrate our results.

\subsection*{Notation and conventions}

Throughout, let $\ell$ be an odd prime and $C_\ell/\Q$ denote the curve $y^\ell=x(x^\ell-1)$. We  write $\zeta_m$ for a primitive $m^{th}$ root of unity. For any rational number $x$ whose denominator is coprime to a positive integer $r$, $\langle x \rangle_r$ denotes the unique representative of $x$ modulo $r$ between 0 and $r-1$. 

Let $I$ denote the $2\times 2$ identity matrix and define the matrix
\begin{align*}
    J=\begin{pmatrix}0&1\\-1&0\end{pmatrix}.
\end{align*}

We embed $\U(1)$ in $\SU(2)$ via $u\mapsto U= \diag(u, \overline u)$. For any positive integer $n$, define the following subgroup of the unitary symplectic group $\USp(2n)$
\begin{equation}\label{eqn:U1n}
    \U(1)^n:=\left\langle \diag( U_1, U_2,\ldots, U_n)\;|\; U_i\in \U(1)\right\rangle.
\end{equation}

We denote the Sato-Tate group of an abelian variety $A/\Q$ by $\ST(A):=\ST(A_{\mathbb Q})$ with identity component denoted $\ST^0(A)$ and component group  $\ST(A)/\ST^0(A)$.

\subsection*{Acknowledgments}
The authors thank the anonymous reviewer for their careful reading of this article and for their helpful comments. This research collaboration initiated through the Tow Mentoring and Research Program at Brooklyn College, City University of New York. We are grateful for this support provided by the Tow Foundation. Our research was also supported by a grant from the National Science Foundation (DMS - 2201085).


\section{Preliminaries on the curve $C_\ell$ and its Jacobian}\label{sec:curveprelim}

The curves $C_\ell$ that we work with in this paper are a special case of those appearing in \cite{Aoki2002}:
$$y^\ell=x^a(x^{\ell^{e-1}}-1).$$
More specifically, we are considering the case where $a=1$ and $e=2$. From the results of  \cite[Corollary 2.2]{Aoki2002} and \cite{KoblitzRohrlich}, we know that the Jacobian variety $\Jac(C_\ell)$ is a simple, nondegenerate abelian variety of dimension $g={\phi(\ell^{2})}/{2}=\ell(\ell-1)/2$, where $\phi$ denotes the Euler totient function. The curve $C_\ell$ has CM by $K:=\Q(\zeta_{\ell^2})$ and the endomorphism ring of $\Jac(C_\ell)$ satisfies $\End(\Jac(C_\ell)_{\overline\Q})\otimes \Q \simeq K$.

For ease of notation, let $\zeta:=\zeta_{\ell^2}$. We define the automorphism $\alpha$ of $C_\ell$ by
\begin{equation}\label{eqn:alpha_automorphism}
    \alpha: (x,y) \mapsto (\zeta^{\ell}x,\zeta^{\ell+1}y).
\end{equation}
One can check that the order of $\alpha$ is $\ell^2$. As we will see in Section  \ref{sec:endomorphism}, the curve automorphism $\alpha$ induces an endomorphism of the Jacobian of $C_\ell$. The rational endomorphism ring of the Jacobian is $\Q(\alpha)\simeq K$.
\begin{example}
    For $\ell=5$, we have the curve $ C_{5}: y^{5} = x(x^{5}-1)$  with curve automorphism
    \begin{equation*}
        \alpha: (x,y) \mapsto (\zeta_{25}^{5}x,\zeta_{25}^{6}y),
    \end{equation*}
    whose order is 25.
\end{example}

\subsection{Endomorphisms of the Jacobian}\label{sec:endomorphism}

As noted above, the dimension of $\Jac(C_\ell)$ is $g=\ell(\ell-1)/2$. It follows from \cite[Problem IV.1.E]{miranda} that we can take the following basis for the space of regular differential 1-forms: 
    \begin{equation}\label{eqn:differential1form}
       \mathcal B=\left\{\omega_{a,b} := x^{a}\frac{dx}{y^{b}} \ \middle|\
        \begin{array}{c}
           0 \leq a \leq \ell-2,\\[.1cm] a+1 \leq b \leq \ell-1
        \end{array} \right\}.
    \end{equation}

\begin{definition}[Ordering Scheme]\label{def:orderingscheme}
    To order the list of 1-forms appearing in Equation \eqref{eqn:differential1form}, fix $a$ and vary $b$ until every $\omega_{a,b}$ has been exhausted for the fixed $a$. Then, organize the list in increasing order (from left to right) of $a$ first, then $b$.
\end{definition}

\begin{example} \label{ex:1formlist}
    For $C_5$, the 1-forms in the basis $\mathcal B$, ordered as in Definition \ref{def:orderingscheme}, are
    \begin{gather*} \omega_{0,1},\,\omega_{0,2},\,\omega_{0,3},\,\omega_{0,4},\,\omega_{1,2},\,\omega_{1,3},\,\omega_{1,4},\,\omega_{2,3},\,\omega_{2,4},\,\omega_{3,4}.
    \end{gather*}
\end{example}

The pullback of the regular 1-form  $\omega_{a,b}$ associated to the map $\alpha$ is given by
    \begin{equation}\label{eqn:alphapullback}
        \alpha^{*}\omega_{a,b} = \zeta^{\ell(a+1-b)-b}\omega_{a,b}.
    \end{equation}
We obtain an endomorphism of $\Jac(C_\ell)$ by taking the symplectic basis of $H_1(\Jac(C_\ell)_\C,\C)$ (with respect to the skew-symmetric matrix $\diag(J,J,\ldots, J)$) corresponding to the basis of regular 1-forms defined in Equation \ref{eqn:differential1form} (written according to the ordering scheme described in Definition \ref{def:orderingscheme}). Let $Z$ denote the $2\times 2$ diagonal matrix
\begin{equation}
Z=\begin{pmatrix}\zeta&0\\0&\overline\zeta\end{pmatrix}.
\end{equation}
The endomorphism $\alpha$ is a diagonal matrix partitioned into $2 \times 2$ blocks such that the $j$th entry, denoted as $\alpha[j,j]$, is
    \begin{equation}\label{eqn:alphaendomorphism}
        \alpha[j,j] = Z^{\ell(a_{j}+1-b_{j})-b_{j}},
    \end{equation}
$0 \leq j \leq g-1$. For ease of notation, let 
    \begin{equation}\label{eqn:e_exponents}
        e_{j} := \langle \ell(a_{j}+1-b_{j})-b_{j} \rangle_{\ell^2}.
    \end{equation}
Then the set of all exponents of $Z$ appearing as entries in $\alpha$ is
    \begin{equation}\label{eqn:S_exponents}
        S:= \left\{ e_{j} \ \middle|\
        \begin{array}{c}
            0 \leq a_{j} \leq \ell-2, \\[.1cm]
            a_{j}+1 \leq b_{j} \leq \ell-1 
        \end{array}
        \right\}.
    \end{equation}

\begin{remark}\label{remark:e_exponent_properties}
Simple arithmetic shows that $e_i\not=e_j$ if $i\not=j$  and that none of the $e_i$ are divisible by $\ell$. Thus, the set $S$ contains $g=\ell(\ell-1)/2$ distinct elements, i.e., $S$ contains exactly half of the elements of $(\mathbb Z/\ell^2\mathbb Z)^*$. Furthermore, $e_i+e_j\not=\ell^2$ for any $i,j$ satisfying the bounds given in the definition of the set $S$.
\end{remark}

\begin{example} \label{ex:p5endo}
We determine $\alpha[4,4]$ for $C_5$. By Example \ref{ex:1formlist}, $j=4$ corresponds to the regular 1-form $\omega_{1,2}$, giving $a_{4} = 1$ and $b_{4} = 2$. Using Equation (\ref{eqn:alphaendomorphism}), 
 \begin{equation*}
    \alpha[4,4] = Z^{5(1+1-2)-2} = Z^{-2} = Z^{23}.
\end{equation*}
Repeating this process for the nine remaining entries gives the endomorphism
\begin{equation*}
    \alpha=\diag(Z^{24}, Z^{18}, Z^{12}, Z^{6}, Z^{23}, Z^{17}, Z^{11}, Z^{22}, Z^{16}, Z^{21}).
\end{equation*}

\end{example}

We conclude this subsection by demonstrating the action of the Galois group $\Gal(K/\Q)$ on the endomorphism $\alpha,$ where $K=\mathbb Q(\zeta_{\ell^2}).$ It is well-known that this Galois group is isomorphic to the cyclic multiplicative group $(\mathbb Z/\ell^2\mathbb Z)^*$. Let $n\in(\mathbb Z/\ell^2\mathbb Z)^*$, and let $\sigma_n$ denote the element of $\Gal(K/\Q)$ mapping $\zeta\mapsto \zeta^n$. For any power $t$ of the matrix $Z$, we have that 
\begin{equation}\label{eqn:sigmaactionZ}
    \prescript{\sigma_{n}}{}{Z^t} = \begin{pmatrix}
    \zeta^{nt} & 0 \\
    0 & \overline{\zeta^{nt}}
    \end{pmatrix} = Z^{nt}.
\end{equation}

Let $0\leq i\leq g-1$ and let $e_i$ be an element of the set $S$ defined in Equation \eqref{eqn:S_exponents}. For ease of notation, let 
\begin{equation}\label{eqn:g_exponents}
    g_{i} := \langle \ell n(a_{i}+1-b_{i})-nb_{i} \rangle_{\ell^2} = \langle ne_{i} \rangle_{\ell^2}.
\end{equation}
Then $g_i$ is an element of $(\mathbb Z/\ell^2\mathbb Z)^*$, and so, as noted in Remark \ref{remark:e_exponent_properties}, we have either $g_i\in S$ or $\ell^2-g_i\in S$. Thus, the action of the Galois automorphism $\sigma_n$ on the endomorphism $\alpha$ can be described by 
    \begin{equation}\label{eqn:sigmaalpha}
        \prescript{\sigma_{n}}{}{\alpha}[i,i] = \begin{cases}
        Z^{g_{i}} & \text{if $g_{i} \in S$} \\[.1in]
        \overline{Z}^{\ell^2-g_{i}} & \text{if $g_{i} \not\in S$},
    \end{cases}
    \end{equation}
so that all exponents that we obtain as a result of this action are in the set $S$.

\begin{example}
For $C_5$, we have that $\langle \sigma_{2} \rangle = \Gal(\Q(\zeta_{25})/\Q)$. Recall from Example \ref{ex:p5endo} that $Z^{23}$ corresponds to $j=4$ and $Z^{16}$ corresponds to $j=8$. By the action defined in Equation (\ref{eqn:sigmaactionZ}),
\begin{equation*}
    g_{0} = \langle 2(e_{0}) \rangle_{25} = \langle 2(24) \rangle_{25} = 23  \quad \text{and} \quad 
    g_{5} = \langle 2(e_{5}) \rangle_{25} = \langle 2(17) \rangle_{25} = 9.
\end{equation*}
Since $23 \in S$ but $9 \not\in S$, it follows that
\begin{equation*}
    \prescript{\sigma_{2}}{}{\alpha}[0,0] = Z^{23} \quad \text{and} \quad \prescript{\sigma_{2}}{}{\alpha}[5,5] = \overline{Z}^{16}
\end{equation*}
by Equation (\ref{eqn:sigmaalpha}). Continuing this process for the eight remaining entries, we see that $\sigma_2$ acts on the endomorphism $\alpha$ as 
\begin{equation*}
    \prescript{\sigma_{2}}{}{\alpha} = \diag(Z^{23}, Z^{11}, Z^{24}, Z^{12}, Z^{21}, \overline{Z}^{16}, Z^{22}, \overline{Z}^{6}, \overline{Z}^{18}, Z^{17}).
\end{equation*}

\end{example}

\subsection{Point count}\label{sec:pointcount}

Let $q$ be a prime distinct from the prime $\ell$ and $\#C_\ell(\mathbb F_q)$ denote the number of $\mathbb F_q$-rational points on the curve $y^\ell=x(x^\ell-1)$. The first two lemmas of this section are generalizations of Propositions 1 and 2 of \cite{HolzapfelNicolae} for the Picard curves $y^3=x(x^3-1)$. 

\begin{lemma}\label{lemma:pointcount1}
    If $q\not\equiv 1 \pmod{\ell}$ then $\#C_\ell(\mathbb F_q)=q+1$.
\end{lemma}
\begin{proof}
If $q\not\equiv 1 \pmod{\ell}$, then the order of the multiplicative group $\mathbb F_q^*$ is not divisible by $\ell$. Hence, the subgroup $\mathbb F_q^{*\ell}$ of $\ell^{th}$ powers in $\mathbb F_q^*$ must be the full group. Thus, for each value of $x$ in $\mathbb F_q^*$, there is exactly one value of $y$ satisfying $y^\ell=x(x^\ell-1)$. Combining this set of solutions with the solution $(x,y)=(0,0)$, we get that there are $q$  solutions $(x,y)$ in $\mathbb F_q\times\mathbb F_q$. There is a single point at infinity, and so we obtain $\#C_\ell(\mathbb F_q)=q+1$ as desired.
\end{proof}

\begin{lemma}\label{lemma:pointcount2}
    If $q\equiv 1 \pmod{\ell}$ but $q\not\equiv 1 \pmod{\ell^2}$ then $\#C_\ell(\mathbb F_q)=q+1$.
\end{lemma}
\begin{proof}
We closely follow the proof of Proposition 2 of \cite{HolzapfelNicolae}. Let $\chi$ be a character of $\mathbb F_q^*$ of order $\ell$. Note that, since $\ell$ is prime, all of the characters of order $\ell$ for $\mathbb F_q^*$ are of the form $\chi^j$, where $j\in\{1, \ldots, \ell-1\}$. We extend each $\chi^j$ to all of $\mathbb F_q$ by setting $\chi^j(0)=0$. Let $\varepsilon$ denote the trivial multiplicative character for $\mathbb F_q^*$ and set $\varepsilon(0)=1$.

It is well-known that the number of solutions to an equation $z^n=a$ is given by
$$N(z^n=a)=\sum_\lambda \lambda(a),$$
where the sum is taken over all characters whose order divides $n$ (see, for example, \cite[Proposition 8.1.5]{IrelandRosen}). Thus,  we can compute the number of solutions to the equation $y^\ell=x(x^\ell-1)$ by computing
\begin{align*}
    N&=\sum_{c\in\mathbb F_q} N(y^\ell=c(c^\ell-1))\\
    &= \sum_{c\in\mathbb F_q} \left(\varepsilon(c(c^\ell-1)) + \sum_{j=1}^{\ell-1} \chi^j(c(c^\ell-1))\right).
\end{align*}
Note that $\varepsilon(c(c^\ell-1))=1$ for all $c\in\mathbb F_q$, and so $\sum_{c\in\mathbb F_q} \varepsilon(c(c^\ell-1))=q$. We will show that $\sum_{c\in\mathbb F_q} \sum_{j=1}^{\ell-1} \chi^j(c(c^\ell-1))=0.$

Since $\ell|(q-1)$ but $\ell^2\nmid (q-1)$, the multiplicative group $\mathbb F_q^*$ equals the internal direct product of its unique subgroup of order $\ell$, which we denote by $\langle g\rangle$, and of its subgroup of order $(q-1)/\ell$, denoted $U_{(q-1)/\ell}$. Thus, every element of $c\in\mathbb F_q^*$ can be uniquely written in the form 
$c=dg^i,$
where $d\in U_{(q-1)/\ell}$ and $i\in\{0, 1, \ldots, \ell-1\}$. Thus, we obtain
\begin{align*}
    \sum_{c\in\mathbb F_q^*} \sum_{j=1}^{\ell-1} \chi^j(c(c^\ell-1))&= \sum_{j=1}^{\ell-1} \sum_{i=1}^{\ell-1} \sum_{d\in U_{(q-1)/\ell}} \chi^j(dg^i((dg^i)^\ell-1))\\
    &=\sum_{j=1}^{\ell-1} \sum_{i=1}^{\ell-1} \sum_{d\in U_{(q-1)/\ell}} \chi^j(g^i)\chi^j(d^{\ell+1}g^{i\ell}-d)\\
    &=\sum_{j=1}^{\ell-1} \sum_{i=1}^{\ell-1} \sum_{d\in U_{(q-1)/\ell}} \chi^j(g^i)\chi^j(d^{\ell+1}-d),
\end{align*}
where the last equality holds since $g$ is an element of order $\ell$, and so $g^{i\ell}=1$. Thus, we can split up the summand as 
\begin{align}\label{eqn:charactersum}
    \sum_{j=1}^{\ell-1} \sum_{i=1}^{\ell-1} \sum_{d\in U_{(q-1)/\ell}} \chi^j(g^i)\chi^j(d^{\ell+1}-d)&=\sum_{j=1}^{\ell-1} \sum_{d\in U_{(q-1)/\ell}} \left(\chi^j(d^{\ell+1}-d) \sum_{i=1}^{\ell-1} \chi^j(g^i)\right).
\end{align}
Recall that, for each $j$, $\chi^j$ is a character of order $\ell$ and so it is nontrivial on the subgroup $\langle g\rangle$ of order $\ell$. Thus, by Proposition 8.1.2 of \cite{IrelandRosen}, $\sum_{i=1}^{\ell-1} \chi^j(g^i)=0$ for each $j$. This implies that the entire expression on the right-hand-side of Equation \eqref{eqn:charactersum} equals 0. Futhermore, when $c=0$ we have $\chi^j(c(c^\ell-1))=\chi^j(0)=0$, for all $1\leq j\leq \ell-1$.  Hence,  $$\sum_{c\in\mathbb F_q} \sum_{j=1}^{\ell-1} \chi^j(c(c^\ell-1))=0,$$ and so $N=q$. This yields the desired result since $\#C_\ell(\mathbb F_q)=N+1=q+1$.
\end{proof}

The remaining case is $q\equiv 1 \pmod{\ell^2}$, which we handle following the techniques used in \cite[Appendix]{LarioSomoza2018} for the Picard curve. Let $\zeta$ be a primitive $\ell^2$ root of unity, and let $\chi_q$ denote the unique character $\chi_q: \mathbb F_q^* \to \mathbb Q(\zeta)^*$ satisfying 
$$\chi_q(x)=x^{(q-1)/\ell^2} \pmod{q}.$$
For ease of notation, we let $\chi:=\chi_q$ for the remainder of this section. For $a,b\in\mathbb Z/\ell^2\mathbb Z$, let $J_q{(a,b)}$ denote the Jacobi sum
\begin{equation}\label{eqn:JacobiSum}
    J_q(a,b):=J_q(\chi^a, \chi^b)=\sum_{x\in\mathbb F_q} \chi^a(x)\chi^b(1-x)
\end{equation}
and define
\begin{equation}\label{eqn:JacobiTrace}
    \Tr_{\mathbb Q(\zeta)/\mathbb Q}(J_q(a,b))=\sum_{s\in(\mathbb Z/\ell^2\mathbb Z)^*} J_q(sa,sb).
\end{equation}

\begin{lemma}\label{lemma:pointcount3}
       If $q\equiv 1 \pmod{\ell^2}$ then $\#C_\ell(\mathbb F_q)=q+1+\Tr_{\mathbb Q(\zeta)/\mathbb Q}(J_q(\ell(\ell-1),1))$. 
\end{lemma}

\begin{proof}
Let $C_\ell': v^{\ell^2}=u(u+1)^{\ell(\ell-1)}.$ There is an isomorphism $\phi: C_\ell \to C_\ell'$ given by  
$$\phi(x,y)=\left(-\frac{1}{x^\ell}, -\frac{y^{\ell-1}}{x^\ell}\right)$$
with inverse $\phi^{-1}: C_\ell' \to C_\ell$ given by 
$$\phi^{-1}(u,v)=\left(-\frac{(u+1)^{\ell-1}}{v^\ell},-\frac{(u+1)^\ell}{v^{\ell+1}}\right).$$
We can easily determine the number of points on $C_\ell'$ using the same techniques as in the proof of Lemma \ref{lemma:pointcount2}. First note that $\sum_{c\in\mathbb F_q}\chi^0(c)\chi^{0}(c+1)=q$. For all $0<b<\ell$, 
\begin{equation}\label{eqn:chisum}
    \sum_{c\in\mathbb F_q} \chi^{b\ell}(c)\chi^{b\ell\cdot \ell(\ell-1)}(c+1)=\sum_{c\in\mathbb F_q} \chi^{b\ell}(c)=0,
\end{equation} 
where the first equality in \eqref{eqn:chisum} holds since $\chi$ is a character of order $\ell^2$ and the second equality holds since $\chi$ is a nontrivial multiplicative character (see \cite[\textsection 8.1]{IrelandRosen}).
Hence,
\begin{align*}
    \#C_\ell'(\mathbb F_q)&=\sum_{c\in\mathbb F_q} N\left(v^{\ell^2}=c(c+1)^{\ell(\ell-1)}\right)\\
    &= \sum_{c\in\mathbb F_q} \left(\varepsilon(c(c+1)^{\ell(\ell-1)}) + \sum_{j=1}^{\ell^2-1} \chi^j(c(c+1)^{\ell(\ell-1)})\right)\\
    &= 1+q+\sum_{c\in\mathbb F_q}\sum_{a\in(\Z/\ell^2\Z)^*} \chi^a(c)\chi^{a\ell(\ell-1)}(c+1)\\
    &= q+1+\sum_{a\in(\Z/\ell^2\Z)^*}\sum_{d\in\mathbb F_q} \chi^a(d-1)\chi^{a\ell(\ell-1)}(d),
\end{align*}
where the third equality holds by \cite[\textsection 8.3, Thm. 1]{IrelandRosen} and the last equality holds by letting $d=c+1$. Since the order of $\chi$ is odd, we have $\chi^a(-1)=1$ for all $a$. Hence,
\begin{align*}
    \#C_\ell'(\mathbb F_q)&=q+1+\sum_{a\in(\Z/\ell^2\Z)^*}\sum_{d\in\mathbb F_q} \chi^a(-1)\chi^{a\ell(\ell-1)}(d)\chi^a(1-d)\\
    &=q+1+\sum_{a\in(\Z/\ell^2\Z)^*}\sum_{d\in\mathbb F_q}\chi^{a\ell(\ell-1)}(d)\chi^a(1-d)\\
    &=q+1+\sum_{a\in(\Z/\ell^2\Z)^*}J_q(a\ell(\ell-1),a)\\
    &=q+1+\Tr_{\mathbb Q(\zeta)/\mathbb Q}(J_q(\ell(\ell-1),1)),
\end{align*}
where the last equality holds by Equation \eqref{eqn:JacobiTrace}. Since $C_\ell'(\mathbb F_q)$ is isomorphic to $C_\ell(\mathbb F_q),$ this proves the desired result.
\end{proof}

These results will be referenced in Section \ref{sec:tableshistograms}.

\section{The Sato-Tate group of $\Jac(C_\ell)$}\label{sec:SatoTateGroup}

In this section we give explicit generators for the Sato-Tate group of the Jacobian of the curve $C_\ell$ for any prime $\ell\geq 3$. Our proof techniques rely on the nondegeneracy of the Jacobians.

\subsection{Preliminaries on nondegeneracy}\label{sec:nondegeneracy}

We first recall some key facts about nondegeneracy that we depend on in the proofs of Proposition \ref{prop:STidcomponent} and Theorem \ref{theorem:STfullgroup}.

Aoki proved in \cite[Corollary 2.2]{Aoki2002} that the Jacobians of the curves $C_\ell$ are stably nondegenerate in the sense of Hazama \cite{Hazama89}, i.e., the Hodge ring of the Jacobian does not contain exceptional cycles and the Hodge group is maximal. By work of Banaszak and Kedlaya (see  \cite[Theorem 6.6]{Banaszak2015}), this implies that we can describe the component group of the Sato-Tate group via the \emph{twisted Lefschetz group}. For an abelian variety $A/F$ of dimension $g$ defined over a number field $F$, the twisted Lefschetz  group, denoted $\TL(A)$, is a closed algebraic subgroup of  $\Sp_{2g}$ defined by
\begin{align*}
\TL(A):=\bigcup_{\tau \in \Gal(\overline{F}/F)} \LL(A)(\tau),
\end{align*}
where $\LL(A)(\tau):=\{\gamma \in \Sp_{2g}\mid \gamma \alpha \gamma^{-1}=\tau(\alpha) \text{ for all }\alpha \in \End(A_{\overline{F}})\otimes \mathbb{Q}\}$. We then have the following result which connects the twisted Lefschetz group to the algebraic Sato-Tate group.

\begin{proposition}\cite{Banaszak2015}\label{prop:AST=TL}
The algebraic Sato-Tate Conjecture holds for nondegenerate abelian varieties $A$ with $\AST(A)=\TL(A)$.
\end{proposition}

When the algebraic Sato-Tate conjecture holds, the Sato-Tate group of $A$ is a maximal compact Lie subgroup of $\AST(A)\otimes_\mathbb{Q} \mathbb{C}$ contained in $\USp(2g)$.

\subsection{Sato-Tate group results}\label{sec:satotategroupresults}

We begin with a result for the identity component of the Sato-Tate group.

\begin{proposition}\label{prop:STidcomponent} The identity component of the Sato-Tate group of $\Jac(C_\ell)$ is 
$$\ST^0(\Jac(C_\ell))\simeq \U(1)^{g}.$$
where $g=\ell(\ell-1)/2$ is the genus of $C_\ell$.
\end{proposition}

\begin{proof}
By definition, $\ST^0(\Jac(C_\ell))$ is a maximal compact subgroup of $\AST^0(\Jac(C_\ell))\otimes_{\mathbb{Q}}  \mathbb{C}.$ Since the Jacobian variety $\Jac(C_\ell)$ is a simple nondegenerate abelian variety with CM, it follows from Proposition \ref{prop:AST=TL} and from techniques used in \cite[Proposition 3.3]{EmoryGoodson2022} that we can take the maximal compact subgroup $\U(1)^g$.
\end{proof}

We now define the matrix $\gamma$ that will appear in Theorem \ref{theorem:STfullgroup}.

\begin{definition}\label{def:GammaMatrix}
    Recall the definitions of $e_j$ and $g_i$ in Equations \eqref{eqn:e_exponents} and \eqref{eqn:g_exponents}, respectively. Define $\gamma$ to be the $2g \times 2g$ block-signed permutation matrix whose $ij^{th}$ block is
    \begin{equation*}
    \gamma[i,j] = \begin{cases}
        I & \text{if $g_{i}=e_{j}$} \\
        J & \text{if $g_{i}=\ell^{2}-e_{j}$} \\
        0 & \text{otherwise,} \\
    \end{cases}
\end{equation*}
where $0 \leq i,j \leq g-1$. The $ij^{th}$ block of the inverse of $\gamma$ is 
    \begin{equation*}
        \gamma^{-1}[i,j] = \begin{cases}
            I & \text{if $g_{j}=e_{i}$} \\
            -J & \text{ if $g_{j}=\ell^{2}-e_{i}$} \\
            0 & \text{otherwise.}
        \end{cases} 
    \end{equation*}
Note the swap of $i$ and $j$ in the definitions of $\gamma$ and its inverse. 
\end{definition}

The main result of this section is Theorem \ref{theorem:STfullgroup} where we give explicit generators of the Sato-Tate group.

\begin{theorem}\label{theorem:STfullgroup}
    The Sato-Tate group of $\Jac(C_\ell)$, up to isomorphism in $\USp(2g)$, is
    \begin{equation}\label{eqn:STisomorphism}
    \ST(\Jac(C_\ell))\simeq\langle\U(1)^{g},\gamma\rangle,
    \end{equation}
where $\gamma$ is the matrix defined in Definition \ref{def:GammaMatrix}.
\end{theorem}

\begin{proof}
    We proved that the identity component for the Sato-Tate group is $\U(1)^{g}$ in Proposition  \ref{prop:STidcomponent}. We now show that the component group is generated by $\gamma$.

We start by proving that $\gamma$ is an element of the twisted Lefschetz group by verifying that $\gamma\alpha\gamma^{-1}= \prescript{\sigma_{n}}{}{\alpha}$, where $\alpha$ is defined in Equation \eqref{eqn:alphaendomorphism}. Note that the $ij^{th}$ block entry of $\gamma\alpha\gamma^{-1}$ can be expressed as 
    \begin{align*}
        \gamma\alpha\gamma^{-1}[i,j] &= \sum_{k=0}^{n}(\gamma\alpha[i,k])\,\gamma^{-1}[k,j] \\
        &= \sum_{k=0}^{n}\left(\sum_{m=0}^{n}\gamma[i,m]\,\alpha[m,k]\right)\gamma^{-1}[k,j].
    \end{align*}
By Equation \eqref{eqn:alphaendomorphism} we have $\alpha[m,k]=0$ unless $m=k$, and so $\sum_{m=0}^{n}\gamma[i,m]\,\alpha[m,k] = \gamma[i,k]\,\alpha[k,k]$. Hence, 
\begin{equation}\label{eqn:gammaalpha_product}
\gamma\alpha\gamma^{-1}[i,j]=\sum_{k=0}^{n}\gamma[i,k]\,\alpha[k,k]\,\gamma^{-1}[k,j].
\end{equation}

Let $\sigma_n$ be a generator of the cyclic group  $\Gal(\Q(\zeta_{\ell^2})/\Q).$ Both $\alpha$ and $\prescript{\sigma_{n}}{}{\alpha}$ are diagonal matrices, and so we first verify that $\gamma\alpha\gamma^{-1}$ is also diagonal. By construction, each row and column of $\gamma$ and of $\gamma^{-1}$ contains exactly one nonzero block entry. Furthermore, the block entry $\gamma[i,k]$ is nonzero if and only if $\gamma^{-1}[k,i]$ is nonzero. Hence, if $i\not=j$ then at least one of $\gamma[i,k]$ or $\gamma^{-1}[k,j]$ equals the zero matrix. Thus, $\gamma\alpha\gamma^{-1}[i,j]=0$ if $i\not=j$ and $\gamma\alpha\gamma^{-1}$ is a diagonal block matrix.

We now consider the diagonal block entries of $\gamma\alpha\gamma^{-1}$. By Equation \eqref{eqn:gammaalpha_product}, 
\begin{align}\label{eqn:gammaalphagamma}
        \gamma\alpha\gamma^{-1}[i,i]   &= \sum_{k=0}^{n}\gamma[i,k]\,\alpha[k,k]\,\gamma^{-1}[k,i].
    \end{align}
Note that $\gamma[i,k]$ is nonzero for exactly one value of $k$, say $k=j$, and there are two possibilities: either  $\gamma[i,j]=I$ or  $\gamma[i,j]=J$.

First suppose  $\gamma[i,j]=I$. By Lemma 4.2 of \cite{GoodsonCatalan} and Equation \eqref{eqn:gammaalphagamma}, we have
$$\gamma\alpha\gamma^{-1}[i,i] = \alpha[j,j] = Z^{e_{j}},$$
with $e_j\in S$. Since $\gamma[i,j]=I$, Definition \ref{def:GammaMatrix} tells us that $g_{i}=e_{j}$. Hence, $ Z^{e_{j}}= Z^{g_{i}}$, which equals $\prescript{\sigma_{n}}{}{\alpha}[i,i]$ by Equation \eqref{eqn:sigmaalpha}.

Now suppose $\gamma[i,j]=J$. By Lemma 4.2 of \cite{GoodsonCatalan} and Equation \eqref{eqn:gammaalphagamma}, we have
$$\gamma\alpha\gamma^{-1}[i,i] = -J\alpha[j,j]J = -JZ^{e_{j}}J=  \overline{Z}^{e_{j}}.$$

Since $\gamma[i,j]=J$, Definition  \ref{def:GammaMatrix} tells us that $e_{j}=\ell^2-g_{i}$. Hence, $ \overline{Z}^{e_{j}}= \overline{Z}^{\ell^2-g_{i}}$. This equals $\prescript{\sigma_{n}}{}{\alpha}[i,i]$ by Equation \eqref{eqn:sigmaalpha} since  $g_i\not\in S$.  Thus, we have shown that $\gamma\alpha\gamma^{-1}[i,j]= \prescript{\sigma_{n}}{}{\alpha}[i,j]$ for all $i$ and $j$. Hence, $\gamma$ is an element of the twisted Lefschetz group.

We now show that $\gamma$ generates the entire component group $\ST(\Jac(C_\ell))/\ST^0(\Jac(C_\ell))$. Since the Jacobian of the curve $C_\ell$ is nondegenerate, we know by 
\cite[Section 6]{Banaszak2015} that  the component group of the Sato-Tate group is isomorphic to $\Gal(K/\Q)=\langle \sigma_n\rangle$, where $K=\Q(\zeta_{\ell^2})$ is the field of definition of the endomorphisms.  Thus, we need to prove that the order of $\gamma$ in $\ST(\Jac(C_\ell))/\ST^0(\Jac(C_\ell))$ equals the order of $\sigma_n$ which equals $\phi(\ell^2)=\ell(\ell-1)$.

We have seen that conjugating the diagonal matrix $\alpha$ by $\gamma$ permutes (and sometimes conjugates) its diagonal block entries consistent with the action of $\sigma_n$ on $\alpha$. Since $\gamma \alpha \gamma^{-1}$ is again a diagonal block matrix, conjugating by $\gamma$ will again just permute (and sometimes conjugate) the diagonal block entries consistent with the action of $\sigma_n$. Hence, $\gamma^d\alpha\gamma^{-d}$ is a diagonal block matrix for any $d$ and we can write $\gamma^d\alpha\gamma^{-d} =\prescript{(\sigma_{n})^d}{}{\alpha}$.

By the construction in Definition \ref{def:GammaMatrix}, the block entries of $\gamma^d$ are of the form $\pm I, \pm J,$ or $0$. Let $d$ be the smallest positive integer such that $\gamma^d\in \ST^0(\Jac(C_\ell))$. Then $\gamma^d$ must be block diagonal, and so $\gamma^d[i,i]=I$ or $-I$ for all $i$. In both cases, this implies that $\gamma^d\alpha\gamma^{-d}[i,i]=\alpha[i,i]$ for all $i$, and so $\gamma^d\alpha\gamma^{-d}=\alpha$. However, we also know that $\gamma^d\alpha\gamma^{-d}=\prescript{(\sigma_{n})^d}{}{\alpha}$, and so we must have $\prescript{(\sigma_{n})^d}{}{\alpha}=\alpha$. Hence $d$ is a multiple of the order of $\sigma_n$, which implies that the order of $\gamma$ in $\ST(\Jac(C_\ell))/\ST^0(\Jac(C_\ell))$ is $\ell(\ell-1)$. 

Hence, $\gamma$ generates the entire component group of the Sato-Tate group $\ST(\Jac(C_\ell))$, and so we have proved the statement of the theorem.

\end{proof}

\section{Moment statistics}\label{sec:momentstats}
Though the generalized Sato-Tate conjecture is known to be true for the Jacobians of the curves we are considering by work of Johansson \cite[Proposition 16]{Johansson2017}, it is still interesting to compare moment statistics coming from the Sato-Tate group and those computed numerically from the normalized $L$-polynomials of the curves. In particular, by computing moment statistics we are able to give a definitive solution to the problem of determining the limiting distribution of the normalized traces and see how much numerical estimates are limited by computational constraints.

Recall that the $n^{th}$ moment of a probability density function is the expected value of the $n^{th}$ powers of the values: $M_n[X]=E[X^n]$. For more background information on moment statistics see, for example, \cite[Section 6.1]{EmoryGoodson2022}, \cite[Sections 4.3, 4.4]{SutherlandNotes}, and \cite[Section 3]{LarioSomoza2018}.

\subsection{Methods of computation} \label{sec:methodsComp}

In this section we describe a variety of techniques for computing the numerical moments statistics associated to the coefficients of the normalized $L$-polynomials of the curves and the theoretical moment statistics coming from the traces of random matrices in the Sato-Tate group. We implemented these techniques in  SageMath \cite{Sage} and used the Python libraries \texttt{SymPy} and \texttt{NumPy} (see \cite{GHgithub}).

Let $C_\ell/\Q$ denote the curve $y^\ell=x(x^\ell-1)$. Recall that for primes $p$ of good reduction, the $a_1$-coefficient of the normalized $L$-polynomial in Equation \eqref{eqn:normalizedLpolynomial} satisfies $a_1=\#C_\ell(\mathbb F_p)$ and, more generally, $a_n$ encodes arithmetic information about the curve over the fields $\mathbb F_p, \mathbb F_{p^2},\ldots, \mathbb F_{p^n}$.  For the numerical moments, we consider the reduction of the curve modulo primes $p$ of good reduction and use the built-in SageMath \cite{Sage} function \texttt{count\_points(n)}, which computes the number of points over the finite fields $\mathbb F_p, \mathbb F_{p^2},\ldots, \mathbb F_{p^n}$ on schemes defined over $\mathbb F_p$. Using the results of Lemmas \ref{lemma:pointcount1} and \ref{lemma:pointcount2}, we are able to speed up our computations by restricting to primes $p\equiv 1 \pmod{\ell^2}$. However, since the genera of our curves are so large, we are limited to a somewhat small bound for $p$. For example, when $\ell=5$ the genus is 10 and we use $p<2^{22}$. Implementing the formula from Lemma \ref{lemma:pointcount3}, unfortunately, did not meaningfully improve computation time.

The theoretical moments are computed from the coefficients of the characteristic polynomial of random conjugacy classes in the Sato–Tate group. By the isomorphism in Equation \eqref{eqn:STisomorphism} of Theorem \ref{theorem:STfullgroup}, we can easily compute the moment sequence by working with matrices in the group $\langle\U(1)^{g},\gamma\rangle$. Throughout, let $U$ be a matrix in $\U(1)^g$ and let $\mu_k$ denote the projection onto the interval $\left[-\binom{2g}{k},\binom{2g}{k}\right]$ of the Haar measure obtained from the Sato-Tate group. We denote the restriction of $\mu_k$ to the component $U\gamma^i$ by ${}^i\mu_k$. In what follows, we describe two different techniques for working with the characteristic polynomials and computing moment statistics.

\subsubsection{The $a_{1}$-Coefficient} \label{sec:a1Umoments}

In this section we describe methods for computing the $a_1$-coefficients of the matrices of the form $\U\gamma^i$ that appear in the isomorphism in Equation \eqref{eqn:STisomorphism}. First note that if $i>0$ then the matrix $\gamma^i$ has no nonzero entries along its diagonal. This can be seen from the definition of $\gamma$ in Definition \ref{def:GammaMatrix} and from the relationship between $\gamma$ and the action of the Galois automorphism $\sigma_n$ described in the proof Theorem \ref{theorem:STfullgroup}. Hence, when $i>0$, the $a_1$-coefficient of the characteristic polynomial of any matrix of the form $\U\gamma^i$ equals 0 and 
$$M_n[{}^i\mu_1]=0.$$

When $i=0$, we simply have the matrix $U\in \U(1)^g$, which is a diagonal matrix of the form
$$U=\diag(u_0, \overline u_0, u_1,\overline u_1, \ldots,  u_{g-1}, \overline u_{g-1}).$$
The $a_1$-coefficient is related to the trace of the matrix, and  we have $
    a_1 = -(\alpha_{0} + \alpha_1+ \ldots + \alpha_{g-1}),$ where $\alpha_{i} := u_{i} + \overline{u}_{i}.$ The moment sequence for each $\alpha_{i}$ is that of the unitary group $\U(1)$
\begin{equation*}
    M[\mu_{\U(1)}]=[1,0,2,0,6,0,20,\ldots].
\end{equation*}
Thus, the $n^{th}$ moment of the $a_1$-coefficient for $\U\gamma^{0}$ is the expected value of the $n^{th}$ power of this sum, which is computed using properties of expected values as follows:
\begin{equation*}\label{eqn:nthmomentcoeff}
    M_{n}[{}^0\mu_1]=\sum_{b_{0}+b_{1}+...+b_{g-1}=n}\binom{n}{b_{0},b_{1},...,b_{g-1}}M_{b_{0}}[\alpha_{0}]M_{b_{1}}[\alpha_{1}]...M_{b_{g-1}}[\alpha_{g-1}].
\end{equation*}
See \cite{GHgithub} for an implementation of this.

\subsubsection{The $a_{2},a_{3},\ldots, a_g$-Coefficients} \label{sec:a23Umoments}
We now describe techniques for computing the remaining coefficients of the characteristic polynomials. Perhaps the most straightforward method to compute these higher traces is to use the built-in \texttt{charpoly()} function in SageMath. In our work, we found that this function sufficed for matrices of the form  $\U\gamma^i$, where $i>0$, but it was prohibitively slow for the matrices $\U\gamma^0$. 

An alternative method for computing the higher traces is to make use of the fact that the coefficients of the characteristic polynomial of a matrix can be expressed in terms of its eigenvalues (see, for example, \cite[Section 13.4-5]{KornKorn1968}). This works particularly well for the diagonal matrices $\U\gamma^0$ since the eigenvalues are exactly the diagonal entries $u_0,\overline u_0, u_1, \overline u_1, \ldots,  u_{g-1}, \overline u_{g-1}$. See \cite{GHgithub} for an implementation of this.

\subsection{Example: $\ell=5$}\label{sec:tableshistograms}

In Table \ref{table:p5a1} we present some of the moment statistics we obtained for the genus 10 curve $C_5: y^5=x(x^5-1)$ using the techniques described above. The numerical moments coming from the $a_1$-coefficient of the normalized $L$-polynomial were computed for primes $p<2^{22}$. Due to the large genus of this curve, computing enough data for the remaining coefficients would require significantly more computing power than was available even when taking into account the point count results in Lemma \ref{lemma:pointcount1} and \ref{lemma:pointcount2}.

The theoretical moments coming from the Sato-Tate group were computed using the techniques described in Section \ref{sec:a1Umoments} (see also \cite{GHgithub}).

\begin{table}[h]
\begin{tabular}{|c|p{2cm}|p{2cm}|p{2cm}|p{2cm}|p{2cm}|}
\hline
 &$M_2$ & $M_4$ & $M_6$ & $M_8$ \\ 
\hline
$a_1$& 1.00479 & 58.4085 & 5363.22 & 646563\\
\hline
$\mu_1$ & 1 & 57 & 5140 & 615545\\
\hline
\end{tabular}
\caption{Table of $a_1$- and $\mu_1$-moments ($p<2^{22}$) for $y^5=x(x^5-1)$.}\label{table:p5a1}
\end{table}

\vspace{-.1in}
In Figure \ref{fig:histogram_p1mod25} we display a histogram of the distribution of the numerical $a_1$-coefficients. Based on the results of Lemmas \ref{lemma:pointcount1} and \ref{lemma:pointcount2}, we restrict to primes $p\equiv 1\pmod{25}$. 

\begin{figure}[h]
        \centering
        \includegraphics[width=.45\textwidth]{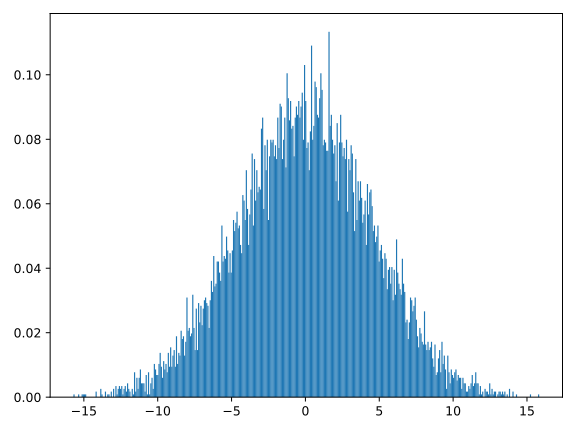} 
        \caption{Histogram of the $a_1$-coefficients for primes $p \equiv 1 \pmod{25}$}\label{fig:histogram_p1mod25}
\end{figure}

\bibliographystyle{alpha}
\bibliography{SatoTatebib}

\end{document}